\documentclass[12pt]{article}
\usepackage[T1]{fontenc}
\usepackage{amsmath}
\usepackage{amssymb}
\usepackage{amsthm}
\usepackage{tikz}
\usetikzlibrary{patterns}
\usepackage{pst-all}

\newtheorem{theorem}{Theorem}[section]
\newtheorem{definitio}[theorem]{Definition}
\newenvironment{definition}{\begin{definitio} \rm }{\end{definitio}}
\newtheorem{lemma}[theorem]{Lemma}
\newtheorem{proposition}[theorem]{Proposition}
\newtheorem{corollary}[theorem]{Corollary}

\DeclareMathOperator{\Cof}{Cof}

\newcommand{\fl}[1]{\textrm{\raisebox{-0.12cm}{\begin{tikzpicture}\draw (0.5,0) node{$\longrightarrow$}; \draw (0.5,0.17) node{$\scriptstyle{#1}$};
  \end{tikzpicture}}}}

\newcommand{\lbp}{\textrm{{\Large (}}}

\newcommand{\rbp}{\textrm{{\Large )}}}
\newcommand{\proofof}[1]{\noindent{\textit{Proof of #1.}}}
\newcommand{\fin}{\hfill$\square$\bigskip}
\newcommand{\Oplus}[1]{\mathop{\bigoplus}\limits_{#1}}

\newcommand{\N}{\mathbb{N}}
\newcommand{\Z}{\mathbb{Z}}
\newcommand{\Zt}{\mathbb{Z}[t^{\pm1}]}

\newcommand{\Q}{\mathbb{Q}}
\newcommand{\Qt}{\mathbb{Q}[t^{\pm1}]}

\newcommand{\CS}{\mathcal{S}}
\newcommand{\Lag}{\mathcal{L}}
\newcommand{\Al}{\mathcal{A}}
\newcommand{\fs}{\underline{f}}
\newcommand{\bs}{\underline{b}}
\newcommand{\etas}{\underline{\eta}}

\title{Rational Blanchfield forms, S-equivalence, and null LP-surgeries}
\author{Delphine Moussard}
\date{ }

\begin{document}

\maketitle

\begin{abstract}
Null Lagrangian-preserving surgeries are a generalization of the Garoufalidis and Rozansky null-moves, that these authors introduced 
to study the Kricker lift of the Kontsevich integral, in the setting of pairs $(M,K)$ 
composed of a rational homology sphere $M$ and a null-homologous knot $K$ in $M$. They are defined as replacements of null-homologous 
rational homology handlebodies of $M\setminus K$ by other such handlebodies with identical Lagrangian. 
A null Lagrangian-preserving surgery induces a canonical isomorphism between the Alexander $\Qt$-modules of the involved pairs, which 
preserves the Blanchfield form. Conversely, we prove that a fixed isomorphism between Alexander $\Qt$-modules which preserves 
the Blanchfield form can be realized, up to multiplication by a power of $t$, by a finite sequence of null Lagrangian-preserving 
surgeries. We also prove that such classes of isomorphisms can be realized by rational S-equivalences. 
In the case of integral homology spheres, we prove similar realization results for a fixed isomorphism between Alexander 
$\Zt$-modules.


\end{abstract}

\tableofcontents

    \section{Introduction}

 \subsection{Context}

In \cite{GR}, Garoufalidis and Rozansky studied the rational vector space generated by the pairs $(M,K)$ modulo orientation-preserving 
homeomorphism, where $M$ is an {\em integral homology 3-sphere} ($\Z$HS), that is an oriented compact 3-manifold which has the same homology 
with integral coefficients as $S^3$, and $K$ is a knot in $M$. They defined a filtration on this space by means of null-moves, that are surgeries 
on claspers (see Garoufalidis, Goussarov and Polyak \cite{GGP}, and Habiro \cite{Hab}) whose leaves are trivial in $H_1(M\setminus K;\Z)$. 
They studied this filtration with the Kricker 
lift of the Kontsevich integral defined in \cite{GK}. The first step in the study of this filtration is the determination of the classes 
of pairs $(M,K)$ up to null-moves. As a corollary of results of Matveev \cite{Mat}, Naik and Stanford \cite{NS}, and Trotter \cite{Trotter}, 
Garoufalidis and Rozansky established that two pairs $(M,K)$ as above can be obtained from one another by a finite sequence of null-moves 
if and only if they admit S-equivalent Seifert matrices, and if and only if they have isomorphic integral Alexander modules and 
Blanchfield forms.

In this article, we consider pairs $(M,K)$, where $M$ is a {\em rational homology 3-sphere} 
($\Q$HS), {\em i.e.} an oriented compact 3-manifold which has the same homology with rational coefficients as $S^3$, 
and $K$ is a {\em null-homologous knot} in $M$, {\em i.e.} a knot whose class in $H_1(M;\Z)$ is trivial. 
We define the null Lagrangian-preserving surgeries, which play the role played by the null-moves in the integral case. 
We prove that the classes of pairs $(M,K)$ modulo null Lagrangian-preserving surgeries are characterized by the classes of rational 
S-equivalence of their Seifert matrices, or by the isomorphism classes of their rational Alexander modules equipped with their 
Blanchfield forms. Furthermore, we prove that a fixed isomorphism between rational Alexander modules which preserves 
the Blanchfield form can be realized, up to multiplication by a power of $t$, by a finite sequence of null Lagrangian-preserving 
surgeries. Null Lagrangian-preserving surgeries define a filtration of the rational vector space generated by pairs $(M,K)$ 
modulo orientation-preserving homeomorphism. This article is a first step in the study of this filtration, that is useful in the study 
of equivariant finite type knot invariants.

In \cite{GR}, Garoufalidis and Rozansky characterized the classes of pairs $(M,K)$, made of a $\Z$HS $M$ and a knot $K\subset M$, 
modulo null-moves, but they did not treat the question of the realization of a fixed isomorphism. In this article, we consider 
integral null Lagrangian-preserving surgeries, which generalize the null-moves, and define the same filtration of the vector 
space generated by all pairs $(M,K)$ modulo orientation-preserving homeomorphism. We prove that a fixed isomorphism between integral 
Alexander modules which preserves the Blanchfield form can be realized, up to multiplication by a power of $t$, by a finite sequence 
of integral null Lagrangian-preserving surgeries. Garoufalidis and Rozansky used their work to determine the graded space associated 
with the above filtration in the case of a trivial Alexander module. In order to study the general case of a maybe non trivial Alexander 
module, the realization result is essential.

When it does not seem to cause confusion, we use the same notation for a curve and its homology class. 

   \subsection{Alexander module and Blanchfield form}

We first recall the definition of the Alexander module and of the Blanchfield form. 
Let $(M,K)$ be a {\em $\Q$SK-pair}, that is a pair made of a rational homology sphere $M$ and a null-homologous knot $K$ in $M$. 
Let $T(K)$ be a tubular neighborhood of $K$. The \emph{exterior} of $K$ is 
$X=M\setminus Int(T(K))$. Consider the natural projection $\pi : \pi_1(X) \to \frac{H_1(X;\Z)}{torsion} \cong \Z$,
and the covering map $p : \tilde{X} \to X$ associated with its kernel. The covering $\tilde{X}$ is the \emph{infinite cyclic covering} 
of $X$. The automorphism group of the covering, $Aut(\tilde{X})$, is isomorphic to $\Z$. It acts on 
$H_1(\tilde{X};\Q)$. Denoting the action of a generator $\tau$ of $Aut(\tilde{X})$ as the multiplication by $t$, 
we get a structure of $\Qt$-module on 
$\Al(K)=H_1(\tilde{X};\Q)$. This $\Qt$-module is the \emph{Alexander module} of $K$. It is a torsion $\Qt$-module. 
\begin{definition}
 Let $(M,K)$ and $(M',K')$ be $\Q$SK-pairs. Let $\xi: \Al(K)\to \Al(K')$ be an isomorphism. 
The {\em $\tau$-class} of $\xi$ is the set of the isomorphisms $\xi\circ m_k$ for $k\in\Z$, where $m_k$ is the multiplication by $t^k$.
\end{definition}
Note that the $\tau$-class of $\xi$ is composed of all the isomorphisms that can be obtained from $\xi$ by composition by isomorphisms 
of $\Al(K)$ or $\Al(K')$ induced by automorphisms of the underlying coverings.

If $(M,K)$ is a {\em $\Z$SK-pair}, {\em i.e.} if $M$ is a $Z$HS, define the {\em integral Alexander module} $\Al_\Z(K)$ as the $\Zt$-module 
$H_1(\tilde{X};\Z)$, similarly. It is a torsion $\Zt$-module, but we will see in Section \ref{secZ} that it has no $\Z$-torsion. 
Hence it can be viewed as a $\Zt$-submodule of $\Al(K)$. Define as above the {\em $\tau$-class} of an isomorphism between integral 
Alexander modules. 

On the Alexander module $\Al(K)$, one can define the \emph{Blanchfield form}, or \emph{equivariant linking pairing},  
$\phi_K : \Al(K)\times\Al(K) \to \frac{\Q(t)}{\Qt}$, as follows. First define the equivariant linking number of two knots.
\begin{definition}
 Let $J_1$ and $J_2$ be two knots in $\tilde{X}$ such that $J_1\cap \tau^k(J_2)=\emptyset$ for all $k\in\Z$.
 Let $\delta(t)$ be the annihilator of $\Al(K)$.
 Then $\delta(\tau)J_1$ and $\delta(\tau)J_2$ are rationally null-homologous links. The \emph{equivariant linking number} of $J_1$ and $J_2$ is 
 $$lk_e(J_1,J_2)=\frac{1}{\delta(t)\delta(t^{-1})}\sum_{k\in\Z}lk(\delta(\tau)J_1,\tau^k(\delta(\tau)J_2))t^k.$$
\end{definition}
One can easily see that $lk_e(J_1,J_2)\in\frac{1}{\delta(t)}\Qt$, $lk_e(J_2,J_1)(t)=lk_e(J_1,J_2)(t^{-1})$, and 
$lk_e(P(\tau)J_1,Q(\tau)J_2)(t)=P(t)Q(t^{-1})lk_e(J_1,J_2)(t)$.
Now, if $\gamma$ (resp. $\eta$) is the homology class of $J_1$ (resp. $J_2$) in $\Al(K)$, define $\phi_K(\gamma,\eta)$ by:
$$\phi_K(\gamma,\eta)=lk_e(J_1,J_2)\ mod\ \Qt.$$
Extend $\phi_K$ to $\Al(K)\times\Al(K)$ by $\Q$-bilinearity. 
The Blanchfield form is hermitian: $$\phi_K(\gamma,\eta)(t)=\phi_K(\eta,\gamma)(t^{-1}) \quad\textrm{ and }\quad 
\phi_K(P(t)\gamma,Q(t)\eta)(t)=P(t)Q(t^{-1})\,\phi_K(\gamma,\eta)(t),$$ for all $\gamma,\eta\in\Al(K)$ and all $P,Q\in\Qt$. 
Moreover, it is non degenerate (see Blanchfield \cite{Bla}) : $\phi_K(\gamma,\eta)=0$ for all $\eta\in\Al(K)$ implies $\gamma=0$.

  \subsection{Seifert matrices}

Let $(M,K)$ be a $\Q$SK-pair. 
Let $\Sigma$ be a {\em Seifert surface} of $K$, {\em i.e.} a compact connected oriented surface in $M$ such that $\partial \Sigma=K$. 
Such a surface exists since $K$ is null-homologous. Let $g$ be the genus of $\Sigma$. Let $(f_i)_{1\leq i\leq 2g}$ be a symplectic basis 
of $H_1(\Sigma;\Z)$, {\em i.e.} a basis such that the matrix of the intersection form in $(f_i)_{1\leq i\leq 2g}$ 
is $-J$, where $J$ is made of blocks
$\begin{pmatrix} 0 & -1 \\ 1 & 0 \end{pmatrix}$ on the diagonal, and zeros elsewhere. 
The {\em Seifert matrix} of $K$ associated with $\Sigma$ and $(f_i)_{1\leq i\leq 2g}$ 
is the matrix $V\in\mathcal{M}_{2g}(\Q)$ defined by $V_{ij}=lk(f_i,f_j^+)$, where $f_j^+$ is a push-off of $f_j$ in the direction 
of the positive normal of $\Sigma$. This matrix satisfies $V-V^t=J$, where $V^t$ denotes the transpose of $V$. Any rational (resp. integral) matrix 
with this property will be called a {\em Seifert matrix} (resp. an {\em integral Seifert matrix}). In Section \ref{sectop}, we prove 
that any such matrix is indeed the Seifert matrix of a $\Q$SK-pair $(M,K)$, and the Seifert matrix of a $\Z$SK-pair if the matrix is integral. 

Given the Seifert matrix $V$, one can compute the Alexander module $\Al(K)$ and the Blanchfield form $\phi_K$. 
Construct a surface $\hat{\Sigma}$ by adding a band to $\Sigma$ along $K$, so that $\hat{\Sigma}$ is homeomorphic to $\Sigma$ 
and contains $\Sigma$ and $K$ in its interior. Let $T(\Sigma)=\hat{\Sigma}\times[-1,1]$ be a tubular neighborhood of $\Sigma$. 
For $1\leq i \leq 2g$, let $e_i\subset(Int(T(\Sigma))\setminus\Sigma)$ be a meridian of $f_i$. 
The module $\Al(K)$ can be presented as: 
$$\Al(K)=\frac{\bigoplus_{1\leq i\leq 2g}\Qt b_i}{\bigoplus_{1\leq j\leq 2g}\Qt\partial S_j},$$ 
where the $b_i$ are lifts of the $e_i$ in the infinite cyclic covering $\tilde{X}$, and the $S_j$ are lifts of the $f_j\times [-1,1]$. 
Set $f_j^+=f_j\times\{1\}$ and $f_j^-=f_j\times\{-1\}$. Assume the $b_i$ are all chosen in the same copy of $M\setminus \Sigma$. 
For $1\leq j\leq 2g$, let $\tilde{f}_j^+$ and $\tilde{f}_j^-$ be lifts of $f_j^+$ and $f_j^-$ in the same copy of $M\setminus \Sigma$ as the $b_i$. 
Assume the $S_j$ are chosen so that $\partial S_j=t\tilde{f}_j^+-\tilde{f}_j^-$. Then $\partial S_j=\sum_{1\leq i\leq 2g} (tV-V^t)_{ij}b_i$, 
hence $tV-V^t$ is a presentation matrix of $\Al(K)$ (see \cite[Chapter 6]{Lick}). 
Moreover, we have $lk_e(\partial S_j,b_k)=(1-t)\delta_{kj}$. Using the expression of $\partial S_j$ in terms of the $b_i$, 
it follows that the form $\phi_K$ is given by $\phi_K(b_i,b_k)=(1-t)((tV-V^t)^{-1})_{ki}\ mod\ \Qt$ (see Kearton \cite[\S 8]{kearton}). 
If $(M,K)$ is a $\Z$SK-pair, then $V$ is integral, and the same construction shows that $tV-V^t$ is a presentation matrix of the $\Zt$-module 
$\Al_\Z(K)$. 

A {\em $\Q$SK-system} is a quintuple $(M,K,\Sigma,\fs,V)$ where $(M,K)$ is a $\Q$SK-pair, $\Sigma$ is a Seifert surface of $K$ in $M$, 
$\fs=(f_i)_{1\leq i\leq 2g}$ is a symplectic basis of $H_1(\Sigma;\Z)$, and $V$ is the associated Seifert matrix. 
Given a $\Q$SK-system, the associated family $(b_i)_{1\leq i\leq 2g}$ of generators of $\Al(K)$ is determined up to 
multiplication of the whole family by $t^k$ for some integer $k$. 
A {\em $\Z$SK-system} $(M,K,\Sigma,\fs,V)$ is a $\Q$SK-system such that $M$ is a $\Z$HS. 

  \subsection{S-equivalence}

\begin{definition}
 A \emph{row enlargement} of a matrix $V\hspace{-2pt}\in\hspace{-0.6pt}\mathcal{M}_{2g}(\Q)$ is a matrix 
$W\hspace{-2pt}=\hspace{-2pt}\begin{pmatrix} 0 & 0 & 0 \\ 1 & x & \rho^t \\ 0 & \rho & V \end{pmatrix}\hspace{-2pt}$, 
where $x\in\Q$ and $\rho\in\Q^{2g}$.
Then the matrix $V$ is a \emph{row reduction} of $W$.
 A \emph{column enlargement} of $V$ is a matrix 
$W=\begin{pmatrix} 0 & -1 & 0 \\ 0 & x & \rho^t \\ 0 & \rho & V \end{pmatrix}$, where $x\in\Q$ and $\rho\in\Q^{2g}$.
Then the matrix $V$ is a \emph{column reduction} of $W$. 
If all the coefficients of the matrices $V$ and $W$ are integers, then the enlargement, or the reduction, is said to be {\em integral}.
\end{definition}

Note that an enlargement or a reduction of a Seifert matrix still is a Seifert matrix. An enlargement of a Seifert matrix corresponds 
to the addition of a tube to the Seifert surface.

\begin{definition}
 A matrix $P\in\mathcal{M}_{2g}(\Q)$ is {\em symplectic} if $PJP^t=J$.
A {\em rational (resp. integral) symplectic congruence} from a matrix $V$ to a matrix $V'$ is an equality $V'=PVP^t$ for some rational 
(resp. integral) symplectic matrix $P$.
\end{definition}
It is more usual to define a symplectic matrix by $P^tJP=J$. However, the two definitions are equivalent, and our choice takes sense 
when we interpret the symplectic matrix involved in a congruence relation beetween Seifert matrices as the matrix of the 
corresponding isomorphism beetween Alexander modules, see Proposition \ref{propcasinv}. 

Note that, since a symplectic matrix has determinant 1, a symplectic rational (resp. integral) matrix is invertible over $\Q$ (resp. $\Z$).

\begin{definition}
 An {\em elementary rational S-equivalence} is an enlargement, a reduction, or a rational symplectic congruence.
Two Seifert matrices are \emph{rationally S-equivalent} if they can be obtained from one another by a finite sequence of elementary 
rational S-equivalences.
\end{definition}
In particular, two Seifert matrices of a $\Q$SK-pair $(M,K)$ are rationally S-equivalent (see \cite[Theorem 8.4]{Lick} 
for the integral case, which easily generalizes). 
In Section \ref{secconservation}, we show that, given two $\Q$SK-systems, a rational S-equivalence between 
their Seifert matrices induces a canonical $\tau$-class of isomorphisms between their Alexander modules preserving the Blanchfield form. 
In Section \ref{secSeq}, we prove the converse:
\begin{proposition} \label{propSeq}
 Let $(M,K,\Sigma,\fs,V)$ and $(M',K',\Sigma',\fs',V')$ be two $\Q$SK-systems. Let $\xi : \Al(K)\to\Al(K')$ 
be an isomorphism which preserves the Blanchfield form. Then $V$ and $V'$ are related by a rational S-equivalence 
which canonically induces the $\tau$-class of $\xi$. 
\end{proposition}

\begin{definition}
 Two Seifert matrices are \emph{semi-integrally S-equivalent} if they can be obtained from one another by a finite sequence of 
enlargements, reductions, and integral symplectic congruences.
\end{definition}

In Section \ref{sectitleSeq}, as a consequence of Lemmas \ref{lemmadelta} and \ref{lemmasymp}, we obtain:
\begin{theorem} \label{thSeq}
 Two Seifert matrices are rationally S-equivalent if and only if they are semi-integrally S-equivalent. 
Furthermore, let $(M,K,\Sigma,\fs,V)$ and $(M',K',\Sigma',\fs',V')$ be two $\Q$SK-systems. Let $\xi : \Al(K)\to\Al(K')$ 
be an isomorphism which preserves the Blanchfield form. Then $V$ and $V'$ are related by a semi-integral S-equivalence 
which canonically induces the $\tau$-class of $\xi$. 
\end{theorem}

In the case of $\Z$SK-systems, for later applications, we need results with only integral coefficients. 
\begin{definition}
 Two Seifert matrices are \emph{integrally S-equivalent} if they can be obtained from one another by a finite sequence of 
integral enlargements, integral reductions, and integral symplectic congruences.
\end{definition}
In Section \ref{secconservation}, we prove that, given two $\Z$SK-systems, an integral S-equivalence between 
their Seifert matrices induces a canonical $\tau$-class of isomorphisms between their integral Alexander modules preserving the Blanchfield form. 
In Section \ref{secZ}, we prove:
\begin{theorem} \label{thSeqZ}
 Let $(M,K,\Sigma,\fs,V)$ and $(M',K',\Sigma',\fs',V')$ be two $\Z$SK-systems. Let $\xi : \Al_\Z(K)\to\Al_\Z(K')$ 
be an isomorphism which preserves the Blanchfield form. Then $V$ and $V'$ are related by an integral S-equivalence 
which canonically induces the $\tau$-class of $\xi$.
\end{theorem}

   \subsection{Lagrangian-preserving surgeries}

\begin{definition}
 For $g\in \mathbb{N}$, a \emph{genus $g$ rational (resp. integral) homology handlebody} ($\Q$HH, resp. $\Z$HH) 
is a 3-manifold which is compact, oriented, and which has the same homology with rational (resp. integral) coefficients 
as the standard genus $g$ handlebody.
\end{definition}
Such a $\Q$HH is connected, and its boundary is necessarily homeomorphic to the standard genus $g$ surface. 
Note that a $\Z$HH is a $\Q$HH. 

\begin{definition}
The \emph{Lagrangian} $\mathcal{L}_A$ of a $\Q$HH $A$ is the kernel of the map 
$$i_*: H_1(\partial A;\Q)\to H_1(A;\Q)$$
induced by the inclusion. Two $\Q$HH's $A$ and $B$ have \emph{LP-identified} boundaries if $(A,B)$ is equipped with a homeomorphism 
$h:\partial A\fl{\cong}\partial B$ such that $h_*(\mathcal{L}_A)=\mathcal{L}_B$.
\end{definition}
The Lagrangian of a $\Q$HH $A$ is indeed a Lagrangian subspace of $H_1(\partial A;\Q)$ 
with respect to the intersection form.

Let $M$ be a $\Q$HS, let $A\subset M$ be a $\Q$HH, and let $B$ be a $\Q$HH whose boundary is LP-identified with $\partial A$.
Set $M(\frac{B}{A})=(M\setminus Int(A))\cup_{\partial A=\partial B}B$. We say that the $\Q$HS 
$M(\frac{B}{A})$ is obtained from $M$ by \emph{Lagrangian-preserving surgery}, or \emph{LP-surgery}.
 
Given a $\Q$SK-pair $(M,K)$, a \emph{null-$\Q$HH} in $M\setminus K$ is a $\Q$HH $A\subset M\setminus K$ such that 
the map $i_* : H_1(A;\Q)\to H_1(M\setminus K;\Q)$ induced by the inclusion has a trivial image.
A \emph{null LP-surgery} on $(M,K)$ is an LP-surgery $(\frac{B}{A})$ such that $A$ is null in $M\setminus K$. 
The $\Q$SK-pair obtained by surgery is denoted by $(M,K)(\frac{B}{A})$. 

Similarly, define {\em integral LP-surgeries}, {\em null-$\Z$HH's}, and {\em integral null LP-surgeries}. The null-moves 
introduced by Garoufalidis and Rozansky in \cite{GR} are defined as null Borromean surgeries. Borromean surgeries are specific 
integral LP-surgeries (see Matveev \cite{Mat}). In \cite[Lemma 4.11]{AL}, Auclair and Lescop proved that two $\Z$HH's whose boundaries 
are LP-identified can be obtained from one another by a finite sequence of Borromean surgeries in the interior of the $\Z$HH's. 
Hence the classes of $\Z$SK-pairs modulo null integral LP-surgeries are exactly the classes of $\Z$SK-pairs modulo null-moves. 

In Section \ref{secconservation}, we prove that a null LP-surgery induces a canonical isomorphism from the Alexander module 
of the initial $\Q$SK-pair to the Alexander module of the surgered $\Q$SK-pair, which preserves the Blanchfield form. 
Conversely, in Section \ref{sectop}, we prove:
\begin{theorem} \label{thLP}
 Let $(M,K)$ and $(M',K')$ be $\Q$SK-pairs. Assume there is an isomorphism $\xi: \Al(K)\to\Al(K')$ which preserves the Blanchfield form. 
Then $(M',K')$ can be obtained from $(M,K)$ by a finite sequence of null LP-surgeries which induces an isomorphism in the $\tau$-class of $\xi$.
\end{theorem}

Similarly, in Section \ref{secconservation}, we prove that an integral null LP-surgery induces a canonical isomorphism from the integral 
Alexander module of the initial $\Z$SK-pair to the Alexander module of the surgered $\Z$SK-pair, which preserves the Blanchfield form. 
In Section \ref{sectop}, we prove:
\begin{theorem} \label{thLPZ}
 Let $(M,K)$ and $(M',K')$ be $\Z$SK-pairs. Assume there is an isomorphism $\xi: \Al_\Z(K)\to\Al_\Z(K')$ which preserves the Blanchfield form. 
Then $(M',K')$ can be obtained from $(M,K)$ by a finite sequence of integral null LP-surgeries which induces an isomorphism in the $\tau$-class 
of $\xi$.
\end{theorem}

We end the article by proving the following proposition in Section \ref{secfin}.
\begin{proposition} \label{propfin}
 There are $\Q$SK-pairs $(M,K)$ and $(M',K')$ that can be obtained from one another by a finite sequence of null LP-surgeries,
but not by a single null LP-surgery.
\end{proposition}
Note that this cannot happen in the case of integral null LP-surgeries. Indeed, as mentioned above, these surgeries can be realized 
by Borromean surgeries, which can be realized in the regular neighborhood of graphs.

\paragraph{Acknowledgements}
I would like to sincerely thank my advisor, Christine Lescop, for her great guidance.

   \section{Conservation of the Blanchfield form} \label{secconservation}

In this section, we prove that null LP-surgeries (Lemma \ref{lemmacons1}) and relations of rational S-equivalence (Lemma \ref{lemmacons2}) 
induce canonical $\tau$-classes of isomorphisms between the Alexander modules which preserve the Blanchfield form. 
We also state similar results in the integral case.

\begin{lemma} \label{lemmacons1}
 Let $(M,K)$ be a $\Q$SK-pair. Let $A$ be a null-$\Q$HH in $M\setminus K$. Let $B$ be a $\Q$HH whose boundary is LP-identified 
with $\partial A$. Set $(M',K')=(M,K)(\frac{B}{A})$. Then the surgery induces a canonical isomorphism $\xi: \Al(K)\to\Al(K')$ 
which preserves the Blanchfield form.
\end{lemma}
\begin{proof}
In this proof, the homology modules are considered with rational coefficients. 
Let $\tilde{X}$ (resp. $\tilde{X}'$) be the infinite cyclic covering associated with $(M,K)$ 
(resp. $(M',K')$). The preimage $\tilde{A}$ of $A$ in $\tilde{X}$ (resp. $\tilde{B}$ of $B$ in $\tilde{X}'$)
is the disjoint union of $\Z$ copies $A_i$ of $A$ (resp. $B_i$ of $B$).

Set $Y=\tilde{X}\setminus Int(\tilde{A})$. The Mayer-Vietoris sequence associated with $\tilde{X}=\tilde{A}\cup Y$ yields 
the exact sequence:
$$H_1(\partial \tilde{A}) \to H_1(\tilde{A})\oplus H_1(Y) \to H_1(\tilde{X}) \to 0.$$
Since $H_1(\partial \tilde{A})\cong H_1(\tilde{A})\oplus(\Qt\otimes\mathcal{L}_A)$, we get 
$\displaystyle H_1(\tilde{X})\cong \frac{H_1(Y)}{\Qt\otimes\mathcal{L}_A}$.
Similarly, $\displaystyle H_1(\tilde{X}')\cong \frac{H_1(Y)}{\Qt\otimes\mathcal{L}_B}$.
Since $\mathcal{L}_A=\mathcal{L}_B$, the Alexander modules $H_1(\tilde{X})$ and $H_1(\tilde{X}')$ are canonically
identified.

Now consider two null-homologous knots $J$ and $J'$ in $\tilde{X}$ that do not meet $\tilde{A}$, and such that $J\cap\tau^k(J')=\emptyset$ 
for all $k\in\Z$. Consider a Seifert surface $\Sigma$ of $J$.
Assume that $\Sigma$ is transverse to $\partial \tilde{A}$ and $J'$. Write $\Sigma=\Sigma_1\cup\Sigma_2$, where 
$\Sigma_1=\Sigma \cap Y$ and $\Sigma_2=\Sigma \cap \tilde{A}$.
Since $J'$ does not meet $\Sigma_2$, the linking number $lk_{\tilde{X}}(J,J')$ is equal to the algebraic intersection number 
$<J',\Sigma_1>$. Now $\partial \Sigma_2$ is an integral linear combination of curves $\alpha_i\in\mathcal{L}_{A_i}$. 
In $\tilde{X}'$, each $\alpha_i$ lies in $\mathcal{L}_{B_i}$, so each $\alpha_i$ has a multiple 
that bounds a surface in $B_i$. Thus, there is a surface $\Sigma_3\subset\tilde{B}$ such that 
$\partial\Sigma_3=n\partial\Sigma_2$ for some integer $n$. We have $nJ=\partial(n\Sigma_1\cup\Sigma_3)$, thus:
$$lk_{\tilde{X}'}(J,J')=\frac{1}{n}<J',n\Sigma_1\cup\Sigma_3>=<J',\Sigma_1>= lk_{\tilde{X}}(J,J').$$
Since any class $\gamma$ in $H_1(\tilde{X})$ has a multiple that can be represented by a knot $J$ in $Y$
such that $P(t).J$ is null-homologous for some $P\in\Qt$, 
the Blanchfield form is preserved.
\end{proof}

The previous proof still works, when $\Q$ is replaced by $\Z$. Therefore:
\begin{lemma} \label{lemmaintA}
 Let $(M,K)$ and $(M',K')$ be $\Z$SK-pairs. Assume $(M',K')$ can be obtained from $(M,K)$ by an integral 
null LP-surgery. Then this surgery induces a canonical isomorphism between their integral Alexander modules which 
preserves the Blanchfield form.
\end{lemma}

\begin{lemma} \label{lemmacons2}
 Let $(M,K,\Sigma,\fs,V)$ and $(M',K',\Sigma',\fs',V')$ be $\Q$SK-systems. If $V$ and $V'$ are rationally S-equivalent, 
then any S-equivalence from $V$ to $V'$ induces a canonical $\tau$-class of isomorphisms from $\Al(K)$ to $\Al(K')$ which 
preserve the Blanchfield form.
\end{lemma}
\begin{proof}
 Let $(b_i)_{1\leq i\leq 2g}$ be a family of generators of $\Al(K)$ associated with $V$. 
Set $W=tV-V^t$. Recall the Blanchfield form $\phi_K$ is given by $\phi_K(b_i,b_j)=(1-t)(W^{-1})_{ji}$. 
Set $b=\begin{pmatrix} b_1 & b_2 & \dots & b_{2g} \end{pmatrix}$, 
and $r=\begin{pmatrix} r_1 & r_2 & \dots & r_{2g} \end{pmatrix}=bW$. We have:
$$\Al(K)=\frac{\bigoplus_{1\leq i\leq 2g} \Qt b_i}{\bigoplus_{1\leq j\leq 2g} \Qt r_j}.$$ 
Define the same notation with primes for the $\Q$SK-pair $(M',K')$. 

First assume that $V'=PVP^t$ for a rational symplectic matrix $P$. 
Note that $W'=PWP^t$. Define a $\Qt$-isomorphism:
$$\begin{array}{cccc} \tilde{\xi} : & \bigoplus_{1\leq i\leq 2g} \Qt b_i & \to & \bigoplus_{1\leq i\leq 2g} \Qt b'_i \\
  & b_i & \mapsto & (b'P)_i \end{array}.$$
We have $\tilde{\xi}(r_j)=(b'PW)_j=(r'(P^t)^{-1})_j$, thus 
$\tilde{\xi}(\bigoplus_{1\leq j\leq 2g} \Qt r_j)=\bigoplus_{1\leq j\leq 2g} \Qt r'_j$. 
Hence $\tilde{\xi}$ induces an isomorphism $\xi : \Al(K) \to \Al(K')$.
Now, we have: 
\begin{eqnarray*}
 \phi_{K'}(\xi(b_i),\xi(b_j)) &=& \phi_{K'}((b'P)_i,(b'P)_j) \\
 &=& \sum_{k,l}p_{ki}p_{lj}(1-t)((W')^{-1})_{lk} \\
 &=& (1-t)\lbp P^t((P^t)^{-1}W^{-1}P^{-1})P\rbp_{ji} \\
 &=& \phi_K(b_i,b_j).
\end{eqnarray*}

It remains to treat the case of an enlargement. Assume $V=\begin{pmatrix} 0 & 0 & 0 \\ 1 & x & \rho^t \\ 0 & \rho & V' \end{pmatrix}$. 
Then: $$W=\begin{pmatrix} 0 & -1 & 0 \\ t & x(t-1) & (t-1)\rho^t \\ 0 & (t-1)\rho & W' \end{pmatrix}.$$ Thus $b_2$ is trivial, 
and $b_1$ is a linear combination over $\Qt$ of the $b_i$ for $3\leq i\leq 2g$. Hence there is an isomorphism 
$(\Al(K),\phi_K)\cong(\Al(K'),\phi_{K'})$ which identifies $b_i$ with $b'_{i-2}$ for $3\leq i\leq 2g$. 
Proceed similarly for a column enlargement. 

Since the families $(b_i)_{1\leq i\leq 2g}$ and $(b'_i)_{1\leq i\leq 2g}$ are determined up to multiplication by a power of $t$, 
we have associated a $\tau$-class of isomorphisms to each elementary S-equivalence. For a general rational S-equivalence, 
just compose the $\tau$-classes associated with the elementary S-equivalences. 
\end{proof}

The previous proof still works, when $\Q$ is replaced by $\Z$. Therefore:
\begin{lemma} \label{lemmacons2Z}
 Let $(M,K,\Sigma,\fs,V)$ and $(M',K',\Sigma',\fs',V')$ be $\Z$SK-systems. If $V$ and $V'$ are integrally S-equivalent, 
then any integral S-equivalence from $V$ to $V'$ induces a canonical $\tau$-class of isomorphisms from $\Al_\Z(K)$ to $\Al_\Z(K')$ which 
preserve the Blanchfield form.
\end{lemma}

    \section{Relating Seifert matrices} \label{secSeq}

In this section, we prove the next proposition, which implies Proposition \ref{propSeq} for invertible Seifert matrices $V$ and $V'$. 
We end the section by deducing the general case. 
\begin{proposition} \label{propcasinv}
 Let $(M,K,\Sigma,\fs,V)$ and $(M',K',\Sigma',\fs',V')$ be two $\Q$SK-systems. Let $\xi : \Al(K)\to\Al(K')$ be an isomorphism 
which preserves the Blanchfield form. Let $\bs=(b_i)_{1\leq i\leq 2g}$ (resp. $\bs'=(b'_i)_{1\leq i\leq 2g}$) be a family of generators 
of $\Al(K)$ (resp. $\Al(K')$) associated with $V$ (resp. $V'$). If $V$ and $V'$ are invertible, then $(\bs)$ and $(\bs')$ are $\Q$-bases 
of $\Al(K)$ and $\Al(K')$ respectively. Let $P$ be the matrix of $\xi$ with respect to the bases $(\bs)$ and $(\bs')$. 
Then $P$ is symplectic and $V'=PVP^t$.
\end{proposition}
This proposition specifies a result of Trotter \cite[Proposition 2.12]{Trotter}. 

\begin{lemma} \label{lemmaQbase}
Let $(M,K,\Sigma,\fs,V)$ be a $\Q$SK-system. Let $\bs=(b_i)_{1\leq i\leq 2g}$ be a family of generators of $\Al(K)$ 
associated with $V$. If $V$ is invertible, then $\bs$ is a $\Q$-basis of $\Al(K)$, 
and the action of $t$ is given by the matrix $V^t\,V^{-1}$ with respect to the basis $\bs$. 
\end{lemma}
\begin{proof}
We have: $$\Al(K)=\frac{\Oplus{1\leq i\leq 2g} \Qt b_i}{\Oplus{1\leq j\leq 2g} \Qt r_j},$$ 
where $r_j=\sum_{1\leq i\leq 2g} (tV-V^t)_{ij}b_i$. Represent the elements of $(\Qt)^{2g}=\Oplus{1\leq i\leq 2g} \Qt b_i$ 
(resp. $\Q^{2g}$) by column vectors giving their coordinates in the basis $\bs$ (resp. in the canonical basis). 
Define a $\Q$-linear map $\varphi : (\Qt)^{2g} \to \Q^{2g}$ by $t^kX \mapsto (V^t\,V^{-1})^kX$ for all vector $X$ with rational coefficients. 
Let us prove that $\varphi$ induces a $\Q$-isomorphism from $\Al(K)$ to $\Q^{2g}$. 

It is easy to see that $\Qt r_j\hspace{-2pt}\subset\hspace{-2pt} ker(\varphi)$ for all $j$. 
Let $u\in ker(\varphi)$. Write $u\hspace{-3pt}=\hspace{-3pt}\sum_{p\leq k\leq q} t^k X_k$, where the $X_k$ are column vectors with rational coefficients. 
Since $\sum_{p\leq k\leq q} (V^t\,V^{-1})^k X_k=0$, we have $X_p=-\sum_{p<k\leq q}(V^t\,V^{-1})^{k-p}X_k$. 
Thus:
\begin{eqnarray*}
	u &=& \sum_{p<k\leq q} (t^k X_k - t^p(V^t\,V^{-1})^{k-p}X_k) \\
	 &=& (tV-V^t).\sum_{p<k\leq q} \sum_{p<i\leq k} t^{i-1}(V^{-1}V^t)^{k-i}V^{-1}X_k
\end{eqnarray*}
Hence $u\in\Oplus{1\leq j\leq 2g} \Qt r_j$, and it follows that $ker(\varphi)=\Oplus{1\leq j\leq 2g} \Qt r_j$. 
\end{proof}

\begin{corollary} \label{corV}
 Let $(M,K,\Sigma,\fs,V)$ be a $\Q$SK-system. Let $\bs=(b_i)_{1\leq i\leq 2g}$ be a family of generators of $\Al(K)$ 
associated with $V$. Assume $V$ is invertible. Let $T$ be the matrix of the multiplication by $t$ in the basis $\bs$. 
Then $V=(I_{2g}-T)^{-1}J$.
\end{corollary}
\begin{proof}
 By Lemma \ref{lemmaQbase}, $T=V^tV^{-1}$. Hence $(I-T)V=V-V^t=J$.
\end{proof}

Consider $\Q(t)$ as the direct sum over $\Q$ of $\Lambda=\Q[t,t^{-1},(1-t)^{-1}]$ and the subspace $\mathcal{E}$ consisting of $0$ and 
all proper fractions with denominator prime to $t$ and $(1-t)$, where proper means that the degree of the fraction numerator 
is strictly lower than the degree of the denominator. Define a $\Q$-linear function $\chi$ on 
$\Q(t)$ by $\chi(E)=E'(1)$ if $E\in \mathcal{E}$ and $\chi(E)=0$ if $E\in\Lambda$. Since $\chi$ vanishes on $\Qt$,
we may also consider it as a function on $\frac{\Q(t)}{\Qt}$.
\begin{lemma} \label{lemmaprev} 
 For $E\in \mathcal{E}$, $\chi((t-1)E)=E(1)$.
\end{lemma}
\begin{proof}
If $F\in\Q(t)$ has denominator prime to $t$ and $1-t$ and numerator of degree less than or equal
to the degree of its denominator, then $F$ can be written as the sum of a rational constant and of an element of $\mathcal{E}$. 
Hence $\chi(F)=F'(1)$. Apply this to $F=(t-1)E$.
\end{proof}

\begin{lemma} \label{lemmaS}
Let $(M,K,\Sigma,\fs,V)$ be a $\Q$SK-system. Let $(b_i)_{1\leq i\leq 2g}$ be a family of generators of $\Al(K)$ 
associated with $V$. Define a matrix $S$ by $S_{ij}=\chi(\phi_K(b_j,b_i))$. Then $S=-(V-V^t)^{-1}=J$. 
\end{lemma}
\begin{proof}
We have $S_{ij}=-\chi((t-1)((tV-V^t)^{-1})_{ij})$. Let $\Delta(t)=\det(tV-V^t)$ be the Alexander polynomial of $(M,K)$. 
Then $\Delta(1)\neq 0$, and since $V$ is invertible, $\Delta(0)\neq 0$ and the degree of $\Delta$ is equal to the size of the matrix $V$. 
Hence, by Lemma \ref{lemmaprev}, $S_{ij}=-((tV-V^t)^{-1})_{ij}(1)$. Conclude with $V-V^t=J$ and $J^{-1}=-J$.
\end{proof}

\proofof{Proposition \ref{propcasinv}}
Let $T$ (resp. $T'$) be the matrix of the action of $t$ in the basis $\bs$ (resp. $\bs'$). 
By Corollary \ref{corV}, $V=(I_{2g}-T)^{-1}J$ and $V'=(I_{2g}-T')^{-1}J$. 
Since $\xi$ preserves the action of $t$, we have $PT=T'P$. 
Since $\xi$ preserves the Blanchfield form, Lemma \ref{lemmaS} implies $J=P^tJP$. 
It follows that $V'=PVP^t$. 
\fin

\begin{lemma} \label{lemmarealmat}
For any matrix $V\in\mathcal{M}_{2g}(\Q)$, with $g>0$, satisfying $V-V^t=J$, there exists a $\Q$SK-pair $(M,K)$ that admits $V$ as a Seifert matrix. 
If $V\in\mathcal{M}_{2g}(\Z)$, then $M$ can be chosen to be $S^3$. 
\end{lemma}
\begin{proof}
Set $V=(v_{ij})_{1\leq i,j\leq 2g}$. 
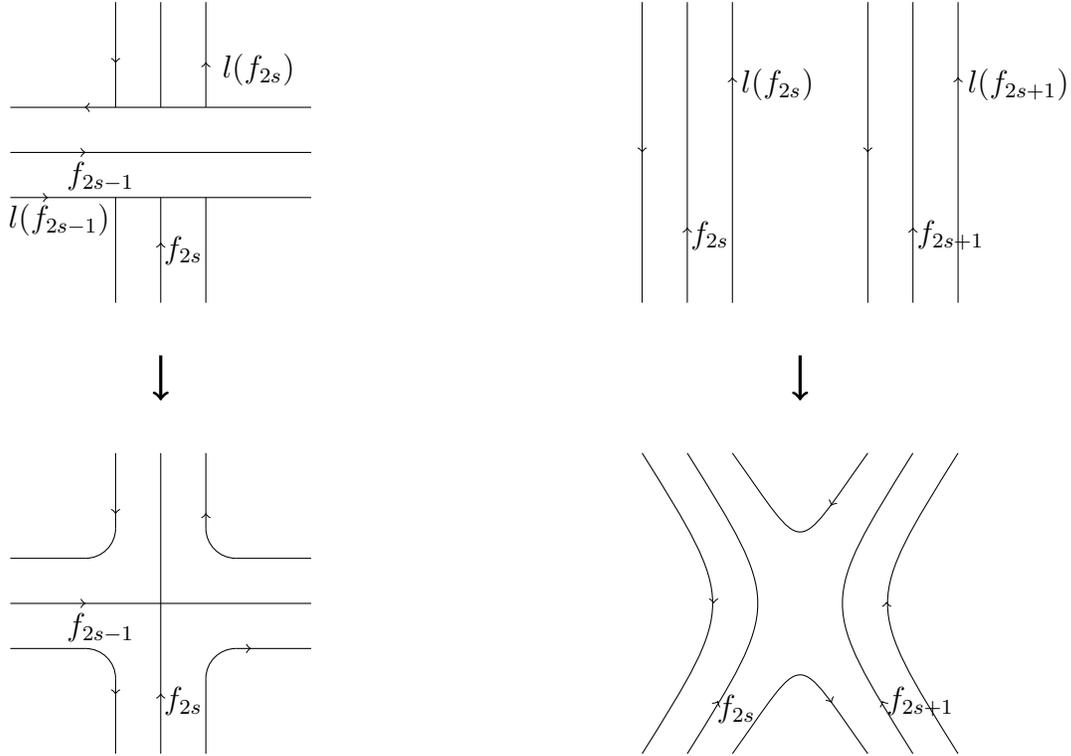
\begin{figure}[htb] 
\begin{center}
\begin{tikzpicture}
\draw (0,8.6) -- (4,8.6);
\draw (0,8) -- (4,8);
\draw (0,7.4) -- (4,7.4);
\draw [->](1,8.6) -- (0.99,8.6);
\draw [->](0.99,8) -- (1,8);
\draw (1.2,7.68) node {$f_{2s-1}$};
\draw [->](0.49,7.4) -- (0.5,7.4);
\draw (0.65,7.1) node {$l(f_{2s-1})$};
      \draw (1.4,6) -- (1.4,7.4);
      \draw (2,6) -- (2,7.4);
      \draw (2.6,6) -- (2.6,7.4);
      \draw (1.4,8.6) -- (1.4,10);
      \draw (2,8.6) -- (2,10);
      \draw (2.6,8.6) -- (2.6,10);
      \draw [->](2,6.79) -- (2,6.8);
      \draw (2.3,6.7) node {$f_{2s}$};
      \draw [->](2.6,9.19) -- (2.6,9.2);
      \draw (3.3,9.1) node {$l(f_{2s})$};
      \draw [->](1.4,9.2) -- (1.4,9.19);
\draw [very thick,->](2,5.3) -- (2,4.7);
\draw (0,2.6) -- (1,2.6);
\draw (1.4,3) -- (1.4,4);
\draw (1,2.6) arc (-90:0:0.4);
\draw [->](1.4,3.2) -- (1.4,3.19);
\draw (2.6,4) -- (2.6,3);
\draw (3,2.6) -- (4,2.6);
\draw (2.6,3) arc (-180:-90:0.4);
\draw [->](2.6,3.19) -- (2.6,3.2);
\draw (0,1.4) -- (1,1.4);
\draw (1.4,1) -- (1.4,0);
\draw (1.4,1) arc (0:90:0.4);
\draw [->](1.4,0.8) -- (1.4,0.79);
\draw (2.6,0) -- (2.6,1);
\draw (3,1.4) -- (4,1.4);
\draw (3,1.4) arc (90:180:0.4);
\draw [->](3.19,1.4) -- (3.2,1.4);
\draw (0,2) -- (4,2);
\draw [->](0.99,2) -- (1,2);
\draw (1.2,1.68) node {$f_{2s-1}$};
\draw (2,0) -- (2,4);
\draw [->](2,0.79) -- (2,0.8);
\draw (2.3,0.7) node {$f_{2s}$};
      \draw (8.4,6) -- (8.4,10);
      \draw (9,6) -- (9,10);
      \draw (9.6,6) -- (9.6,10);
      \draw [->](8.4,8.01) -- (8.4,8);
      \draw [->](9,6.99) -- (9,7);
      \draw (9.3,6.9) node {$f_{2s}$};
      \draw [->](9.6,8.99) -- (9.6,9);
      \draw (10.2,8.9) node {$l(f_{2s})$};
\draw (11.4,6) -- (11.4,10);
\draw (12,6) -- (12,10);
\draw (12.6,6) -- (12.6,10);
\draw [->](11.4,8.01) -- (11.4,8);
\draw [->](12,6.99) -- (12,7);
\draw (12.5,6.9) node {$f_{2s+1}$};
\draw [->](12.6,8.99) -- (12.6,9);
\draw (13.4,8.9) node {$l(f_{2s+1})$};
\draw [very thick,->](10.5,5.3) -- (10.5,4.7);
      \draw (9.6,4) .. controls (10.6,2.6) and (10.4,2.6) .. (11.4,4);
      \draw (9.6,0) .. controls (10.6,1.4) and (10.4,1.4) .. (11.4,0);
      \draw (9,4) .. controls (10.25,2) .. (9,0);
      \draw (8.4,4) .. controls (9.65,2) .. (8.4,0);
      \draw (12,4) .. controls (10.75,2) .. (12,0);
      \draw (12.6,4) .. controls (11.35,2) .. (12.6,0);
      \draw [->](9.4,0.68) -- (9.41,0.7);
      \draw (9.65,0.6) node {$f_{2s}$};
      \draw [->](11.6,0.68) -- (11.59,0.7);
      \draw (12.1,0.7) node {$f_{2s+1}$};
      \draw [->](10.91,0.7) -- (10.92,0.68);
      \draw [->](10.92,3.32) -- (10.91,3.3);
      \draw [->](9.35,2.01) -- (9.35,2);
      \draw [->](11.65,2.01) -- (11.65,2.02);
\end{tikzpicture}
\end{center} \caption{Gluing bands} \label{figbands}
\end{figure}
By \cite[Corollary 2.13]{M1}, there is a $\Q$HS $M$ and pairwise disjoint simple closed framed curves 
$f_i$, $1\leq i\leq 2g$, in $M$, such that $lk(f_i,f_j)=v_{ij}$ for $j\leq i$. 
Consider bands around the $f_i$, that are images of embeddings 
$h_i : [-1,1]\times S^1 \hookrightarrow M$ such that $h_i(\{0\}\times S^1)=f_i$,
and $\ell(f_i)=h_i(\{1\}\times S^1)$ is the parallel of $f_i$ such that $lk(f_i,\ell(f_i))=v_{ii}$.
Connecting these bands as indicated in Figure \ref{figbands}, we get a surface bounded by a knot $K$ which satisfies 
the required conditions. 
\end{proof}

\proofof{Proposition \ref{propSeq}}
If $V$ is non invertible, there exists $g_1\in H_1(\Sigma;\Z)$ such that $lk(g_1,\gamma^+)=0$ 
for all $\gamma\in H_1(\Sigma;\Z)$. Choose for $g_1$ a {\em primitive} element of $H_1(\Sigma;\Z)$, {\em i.e.} 
such that $g_1=kg$ with $k\in\Z$ and $g\in H_1(\Sigma;\Z)$ implies $k=\pm 1$. In any symplectic basis of $H_1(\Sigma;\Z)$, 
$g_1$ has coprime coefficients, hence there is $g_2\in H_1(\Sigma;\Z)$ such that $\langle g_1,g_2\rangle_{\Sigma}=1$, 
where $\langle .,.\rangle_{\Sigma}$ denotes the intersection form on $\Sigma$. Consider a symplectic basis 
$(g_i)_{3\leq i\leq 2g}$ of the orthogonal of $\Z g_1\oplus\Z g_2$ in $H_1(\Sigma;\Z)$ 
with respect to the intersection form. Then $(g_i)_{1\leq i\leq 2g}$ is a symplectic basis of $H_1(\Sigma;\Z)$, 
and the associated Seifert matrix $V_1$ is a row enlargement of a Seifert matrix $V_2$. Since $V$ and $V_1$ are associated with 
the same Seifert surface, they are related by a change of basis of $H_1(\Sigma;\Z)$, {\em i.e.} they are congruent. 
Hence $V$ is rationally S-equivalent to the smaller matrix $V_2$. Iterating this process, we see that $V$ is rationally S-equivalent to an 
invertible Seifert matrix $W$, where we consider that there exists an empty matrix, which is invertible. 
Similarly, $V'$ is rationally S-equivalent to an invertible Seifert matrix $W'$. The matrices $W$ and $W'$ are invertible 
Seifert matrices that define isomorphic Blanchfield forms. 

By Lemma \ref{lemmarealmat}, there are $\Q$SK-systems $(N,J,S,\etas,W)$ and $(N',J',S',\etas',W')$. 
The rational S-equivalence relation between $V$ and $W$ (resp. $V'$ and $W'$) induces the $\tau$-class of an isomorphism $\zeta: \Al(J)\to\Al(K)$ 
(resp. $\zeta': \Al(J')\to\Al(K')$). By Proposition \ref{propcasinv}, there is an invertible rational symplectic matrix $P$ such that 
$W'=PWP^t$ and $P$ induces the $\tau$-class of the isomorphism $(\zeta')^{-1}\circ\xi\circ\zeta:\Al(J)\to\Al(J')$.
\fin

    \section{Rational S-equivalence} \label{sectitleSeq}

In this section, we prove Theorem \ref{thSeq} by proving that we can realize a symplectic rational congruence by a finite 
sequence of enlargements, reductions, and integral symplectic congruences which, for given $\Q$SK-systems, induces the same $\tau$-class of 
isomorphisms between the Alexander modules as the initial congruence. We first treat a particular type of congruence matrices.

\begin{lemma} \label{lemmadelta}
 Let $V$ and $W$ be two Seifert matrices such that $\Delta_nV\Delta_n=W$, where $n$ or $\frac{1}{n}$ is a positive integer,
and $$\Delta_n=\begin{pmatrix} n &&&& \\ & \frac{1}{n} && 0 & \\ && 1 && \\ & 0 && \ddots & \\ &&&& 1 \end{pmatrix}.$$
Then there are enlargements $\tilde{V}$ of $V$ and $\tilde{W}$ of $W$ that are related by an integral symplectic congruence. 
Furthermore, if $(M,K,\Sigma,\fs,V)$ and $(M',K',\Sigma',\fs',W)$ are two $\Q$SK-systems, 
$W$ can be obtained from $V$ by a sequence of an enlargement, an integral symplectic congruence, and a reduction, 
which induces the same $\tau$-class of isomorphisms from $\Al(K)$ to $\Al(K')$ as the congruence matrix $\Delta_n$.
\end{lemma}
\begin{proof}
Assume $n$ is a positive integer.

Set $\displaystyle V=\begin{pmatrix} p & q & \omega^t \\ r & s & \rho^t \\ \omega & \rho & U \end{pmatrix}$. Then 
$\displaystyle W=\begin{pmatrix} n^2p &q& n\omega^t \\ r&\frac{s}{n^2} &\frac{1}{n}\rho^t \\ n\omega & \frac{1}{n}\rho & U \end{pmatrix}$. 
Note that $r=q+1$. 
Set: $$P=\left(\begin{array}{c|c} \begin{array}{cccc} 0 & n & 0 & -1 \\ 0 & 0 & 1 & 0 \\ 1 & 0 & n & 0 \\ 0 & 1 & 0 & 0 \end{array} &
\quad 0 \quad \\ \hline & \\ 0 & I \\ &  \end{array}\right) , \quad \tilde{V}=\left(\begin{array}{cccc|c} 0 & -1 & 0 &0& \quad 0 \quad \\ 
0 & \frac{s}{n^2} & \frac{r}{n} & \frac{s}{n} & \frac{1}{n}\rho^t \\ 0 & \frac{r}{n} & p & q & \omega^t \\ 
0 & \frac{s}{n} & r & s & \rho^t \\ \hline &&&& \\ 0 & \frac{1}{n}\rho & \omega & \rho & U \\ &&&& \end{array} \right) ,$$
$$\textrm{and } \tilde{W}=\left(\begin{array}{cccc|c} 0 & 0 & 0 & 0 & \quad 0 \quad \\ 1 & p & np & \frac{r}{n} & \omega^t \\
0 & np & n^2p & q & n\omega^t \\ 0 & \frac{r}{n} & r & \frac{s}{n^2} & \frac{1}{n}\rho^t \\ \hline &&&& \\ 
0 & \omega & n\omega & \frac{1}{n}\rho & U \\ &&&& \end{array} \right) . $$
The matrix $P$ is integral and symplectic, and we have $\tilde{W}=P\tilde{V}P^t$.

Let $(b_i)_{1\leq i\leq 2g}$ (resp. $(b'_i)_{1\leq i\leq 2g}$) be a family of generators of $\Al(K)$ (resp. $\Al(K')$) 
associated with $V$ (resp. $W$). 
The congruence matrix $\Delta_n$ induces the $\tau$-class of the isomorphism $\xi: \Al(K)\to\Al(K')$ such that 
$\xi(b_i)=(\begin{pmatrix} b'_1 & \dots & b'_{2g}\end{pmatrix}\Delta_n)_i$. 
Let us check that the obtained sequence of an enlargement, an integral symplectic congruence, and a reduction, also induces $\xi$. 
The matrix $\tilde{V}$ defines a $\Qt$-module $\Al\cong\Al(K)$ with a generating family $(\tilde{b}_i)_{1\leq i\leq 2g+2}$ 
and relations $(\begin{pmatrix} \tilde{b}_1 & \dots & \tilde{b}_{2g+2}\end{pmatrix}(t\tilde{V}-\tilde{V}^t))_j$. 
The enlargement of $V$ into $\tilde{V}$ induces the $\tau$-class of the isomorphism $\zeta: \Al(K)\to\Al$ such that $\zeta(b_i)=\tilde{b}_{i+2}$. 
Similarly, define $\Al'$, $(\tilde{b}'_i)_{1\leq i\leq 2g+2}$ and $\zeta'$. The congruence matrix $P$ induces the $\tau$-class of the 
isomorphism $\vartheta: \Al\to\Al'$ such that $\vartheta(\tilde{b}_i)=(\begin{pmatrix} \tilde{b}'_1 & \dots & \tilde{b}'_{2g+2}\end{pmatrix}P)_i$. 
Set $\xi'=(\zeta')^{-1}\circ\vartheta\circ\zeta$. 
Let us check that $\xi'(b_i)=\xi(b_i)$ for all $i\in\{1,..,2g\}$. For $i\geq 3$, it is obvious. For $i=1$, it follows from 
the relation $\tilde{b}'_2=0$ given by the first column of the matrix $t\tilde{W}-\tilde{W}^t$. For $i=2$, we have 
$\xi'(b_2)=(\zeta')^{-1}(-\tilde{b}'_1)$. Since the second column of $t\tilde{W}-\tilde{W}^t$ gives:
$$\tilde{b}'_1=(t-1)\lbp np\tilde{b}'_3+\frac{r}{n}\tilde{b}'_4+\begin{pmatrix} \tilde{b}'_5 & \dots & \tilde{b}'_{2g+2}\end{pmatrix}\omega\rbp,$$
and since the first column of $tW-W^t$ gives:
$$b'_2=-(t-1)\lbp n^2 p b'_1+rb'_2+n\begin{pmatrix} b'_3 & \dots & b'_{2g}\end{pmatrix}\omega\rbp,$$
we have $\xi'(b_2)=\frac{1}{n}b_2'$.

Since $\Delta_{\frac{1}{n}}=\Delta_n^{-1}$, the case $\frac{1}{n}\in\mathbb{N}\setminus\{0\}$ follows.
\end{proof}

\begin{lemma} \label{lemmasymp}
 Any symplectic rational matrix $P$ can be written as a product of integral symplectic matrices and matrices $\Delta_n$
or $\Delta_\frac{1}{n}$ for positive integers $n$.
\end{lemma}
\begin{proof} \ 
\paragraph{Step 1:} There is no loss in assuming that the first column of $P$ is $\begin{pmatrix} 1 \\ 0 \\ \vdots \\ 0 \end{pmatrix}$.

Denote by $d$ a common denominator for the terms of the first column of $P$. The matrix $P\Delta_d$ has integral coefficients
in its first column. Denote by $\delta$ their gcd. The terms of the first column of $P\Delta_d\Delta_{\frac{1}{\delta}}$
are coprime integers. There is an integral symplectic matrix $Q$ with the same first column. The matrix 
$Q^{-1}P\Delta_d\Delta_{\frac{1}{\delta}}$ has the required first column.

\paragraph{Step 2:} We can assume that the first two columns of $P$ are 
$\begin{pmatrix} 1 & 0 \\ 0 & 1 \\  \vdots & 0 \\ \vdots & \vdots \\ 0 & 0 \end{pmatrix}$.

The matrix $P^{-1}$ has the same first column as $P$. Since it is symplectic, its second column is 
$\begin{pmatrix} x_1 \\ \vdots \\ x_{2g} \end{pmatrix}$, with $x_2=1$. Set:
$$Q=\begin{pmatrix} 1 & x_1 & -x_4 & x_3 & \dots & -x_{2g} & x_{2g-1} \\ 0 & 1 & 0 & \dots & \dots & \dots & 0 \\
   0 & x_3 & 1 & & & & \\ \vdots & \vdots & & \ddots & & 0 & \\ \vdots & \vdots & & & \ddots & & \\ 
   \vdots & \vdots & & 0 & & \ddots & \\ 0 & x_{2g} & & & & & 1 \end{pmatrix}.$$
Since $Q$ has the same first two columns as $P^{-1}$, the matrix $PQ$ has the required first two columns.
Now, if $n$ is a common denominator for all the $x_i$, the matrix $\Delta_nQ\Delta_\frac{1}{n}$ has integral 
coefficients, and is symplectic.

\paragraph{Step 3:} Induction.

We have $P=\begin{pmatrix} I_2 & R \\ 0 & Q \end{pmatrix}$. Since $P$ is symplectic, $R=0$ and $Q$ is symplectic.
Thus we can conclude by induction on $g$.\end{proof}

    \section{Relating integral Seifert matrices} \label{secZ}

In order to prove Theorem \ref{thSeqZ}, we want to proceed as for proving Theorem \ref{thSeq}. Here, we have to avoid enlargements 
with non integral coefficients. Thus we shall be more careful in the way we decompose rational symplectic congruences. 
Following Trotter \cite{Trotter}, we introduce some formalism. Set $z=(1-t)^{-1}$. 

\begin{definition}
 A {\em scalar space} $\Al$ is a finitely generated torsion $\Q[t,t^{-1},z]$-module, endowed with a {\em scalar form} $[.,.]:\Al\times\Al\to\Q$, 
that is a $\Q$-bilinear non-degenerate anti-symmetric form which satisfies:
\begin{itemize}
 \item $[ta_1,ta_2]=[a_1,a_2]$,
 \item $[za_1,a_2]=-[a_1,tza_2]=[a_1,(1-z)a_2]$,
\end{itemize}
for all $a_1,a_2\in\Al$.
\end{definition}

Given a $\Q$SK-pair $(M,K)$, define a scalar space structure on the Alexander module $\Al(K)$ as follows. 
Multiplication by $(1-t)$ is an isomorphism of $\Al(K)$; define the action of $z$ as its inverse. 
Define a scalar form on $\Al(K)$ by $[a_1,a_2]=\chi(\phi_K(a_1,a_2))$ for all $a_1,a_2\in\Al$, where $\chi$ is the map 
defined before Lemma \ref{lemmaprev}. 

\begin{definition}
 Let $(\Al,[.,.])$ be a scalar space of $\Q$-dimension $2g$. A {\em lattice} in $\Al$ is a free $\Z$-submodule of rank $2g$. 
Such a lattice $\Gamma$ is {\em self-dual} if, for $a\in\Al$, $a$ is in $\Gamma$ if and only if $[a,x]\in\Z$ for all $x\in\Gamma$. 
A lattice $\Gamma$ in $\Al$ is {\em admissible} if it is self-dual and if it satisfies $z\Gamma\subset\Gamma$. 
\end{definition}

\begin{lemma} \label{lemmalattice}
 Let $(M,K,\Sigma,\fs,V)$ be a $\Z$SK-system such that $V$ is invertible over $\Q$. Let $\bs=(b_i)_{1\leq i\leq 2g}$ be a basis 
of $\Al(K)$ associated with $V$. Let $\Gamma$ be the lattice generated by the basis $\bs$ in $\Al(K)$. 
Then $\Gamma$ is admissible.
\end{lemma}
\begin{proof}
By Lemma \ref{lemmaS}, the matrix of the scalar form on $\Al(K)$, with respect to the basis $\bs$, is $-J$. It follows that $\Gamma$ is self-dual. 
Let $Z$ be the matrix of the action of $z$ in the basis $\bs$. By Corollary \ref{corV}, $Z=-VJ$. Thus $Z$ has integral coefficients. 
\end{proof}

The next lemma implies that $\Al_\Z(K)=\Zt\Gamma\subset\Al(K)$, with the 
notation of Lemma \ref{lemmalattice}. Note that multiplication by $(1-t)$ is an isomorphism of $\Al_\Z(K)$, hence $\Al_\Z(K)=\Lambda\Gamma$, 
where $\Lambda=\Z[t,t^{-1},z]$. 

\begin{lemma}
 The integral Alexander module associated with a $\Z$SK-pair has no $\Z$-torsion.
\end{lemma}
\begin{proof}
 Let $(M,K,\Sigma,\fs,V)$ be a $\Z$SK-system. Let $\bs=(b_i)_{1\leq i\leq 2g}$ be a family of generators of $\Al_\Z(K)$ 
associated with $V$. 
We have: $$\Al_\Z(K)=\frac{\Oplus{1\leq i\leq 2g} \Zt b_i}{\Oplus{1\leq j\leq 2g} \Zt r_j},$$ 
where $r_j=\sum_{1\leq i\leq 2g} W_{ij}b_i$ and $W=tV-V^t$. Let $p: \Oplus{1\leq i\leq 2g} \Zt b_i \twoheadrightarrow \Al_\Z(K)$ be 
the natural projection. Represent the elements of $\Oplus{1\leq i\leq 2g} \Zt b_i$ by column vectors giving their coordinates in the basis $\bs$. 

Let $A\in \Oplus{1\leq i\leq 2g} \Zt b_i$. Assume $kA\in\ker(p)$ for a non trivial integer $k$. Then there is $X\in\Oplus{1\leq i\leq 2g} \Zt b_i$ 
such that $kA=WX$. Thus: $$k\Cof(W)A=\det(W)X=\Delta(t)X,$$ where $\Cof(W)$ is the cofactor matrix of $W$ and $\Delta(t)$ is the Alexander 
polynomial of $(M,K)$. Hence $k$ divides each coefficient of $\Delta(t)X$. Since $\Delta(1)=1$, it implies that $X=kY$ 
for some $Y\in\Oplus{1\leq i\leq 2g} \Zt b_i$. Thus $A=WY\in\ker(p)$.
\end{proof}

Let $\Gamma$ be a lattice in a scalar space $\Al$. A basis $\bs=(b_i)_{1\leq i\leq 2g}$ of $\Gamma$ is {\em symplectic} 
if the matrix of the scalar form with respect to $\bs$ is $-J$. 
\begin{lemma} \label{lemmaselfdual}
 Let $(\Al,[.,.])$ be a scalar space. Let $\Gamma$ be a lattice in $\Al$. Then $\Gamma$ is self-dual if and only if it has a symplectic basis.
\end{lemma}
\begin{proof}
 Assume $\Gamma$ is self-dual. Let $\bs=(b_i)_{1\leq i\leq 2g}$ be any basis of $\Gamma$. 
Set $s_i=[b_1,b_i]$. The self-duality condition implies that the non trivial $s_i$ are coprime. Hence there are integers $u_i$ such that 
$\sum_{i=1}^{2g} u_i s_i=1$. Set $b_2'=\sum_{i=1}^{2g} u_i b_i$, so that $[b_1,b_2']=1$. 
Let $\mathcal{B}$ be the orthogonal in $\Al$ of $\Q b_1\oplus\Q b_2'$ with respect to the scalar form. For $x\in\Gamma$, 
$y=x-[x,b_2']b_1-[b_1,x]b_2'\in\Gamma$ is orthogonal to $b_1$ and $b_2'$. It follows that $\Gamma$ is the direct sum, 
orthogonal with respect to the scalar form, of $\Z b_1\oplus\Z b_2'$ and $\mathcal{B}\cap\Gamma$. 
Conclude by induction on $g$ that $\Gamma$ has a symplectic basis. The reverse implication is easy. 
\end{proof}

\begin{lemma}
 Let $(\Al,[.,.])$ be a scalar space of $\Q$-dimension $2g$. Let $\Gamma$ be an admissible lattice in $\Al$. Let $\bs$ be a symplectic basis 
of $\Gamma$. Let $Z$ be the matrix of the action of $z$ in the basis $\bs$. Set $V=ZJ$. 
Then $V$ is an integral Seifert matrix. 
\end{lemma}
\begin{proof}
 By definition of a scalar form, we have $Z^tJ=J(I-Z)$. It follows that $V-V^t=J$. 
Integrality follows from the admissibility condition.
\end{proof}
The matrix $V$ defined in the above lemma is the {\em Seifert matrix associated with $\Gamma$ and $\bs$}. 

\begin{definition}
 Let $(\Al,[.,.])$ be a scalar space of $\Q$-dimension $2g$. For $n\in\N\setminus\{0\}$, two lattices $\Gamma$ and $\Gamma'$ 
in $\Al$ are {\em $n$-adjacent} if $\displaystyle\frac{\Gamma}{\Gamma\cap\Gamma'}\cong\frac{\Z}{n\Z}$ 
and $\displaystyle\frac{\Gamma'}{\Gamma\cap\Gamma'}\cong\frac{\Z}{n\Z}$. 
Two lattices $\Gamma$ and $\Gamma'$ are {\em adjacent} if they are $n$-adjacent for some $n\in\N\setminus\{0\}$. 
\end{definition}

The following proposition is the object of Section 3 in \cite{Trotter}, althought it is not stated like this. 
\begin{proposition}[Trotter] \label{propTrotter}
 Let $(\Al,[.,.])$ be a scalar space. Let $\Gamma$ and $\Gamma'$ be admissible lattices in $\Al$ such that 
$\Lambda\Gamma=\Lambda\Gamma'$. Then there is a sequence of admissible lattices $\Gamma_0=\Gamma, \Gamma_1, \dots, \Gamma_k=\Gamma'$, 
such that, for $1\leq i\leq k$, $\Gamma_{i-1}$ and $\Gamma_i$ are adjacent, and $z\Gamma_{i-1}\subset\Gamma_i$ 
or $z\Gamma_i\subset\Gamma_{i-1}$.
\end{proposition}
This result implies that two integral Seifert matrices, invertible over $\Q$, which define the same integral Alexander module, 
can be related by integral congruences, and congruences with congruence matrices $\Delta_n$ (see Lemma \ref{lemmadelta} for the definition 
of $\Delta_n$). The proposition says more, with the last condition on the lattices, and we will use this to prove that 
these $\Delta_n$-congruences can be realized by integral S-equivalences. We first check that we can always use symplectic bases 
of the lattices.

An element $a$ of a lattice $\Gamma$ is {\em primitive} if the equality $a=kb$ with $k\in\Z$ and $b\in\Gamma$ implies $k=\pm 1$. It is easy to 
see that it is necessary and sufficient for $a$ to be primitive that the non trivial coefficients of $a$ in any basis of $\Gamma$ are coprime. 
It is also easy to see that, if $\Gamma$ is a self-dual lattice in a scalar space $(\Al,[.,.])$, then $a\in\Gamma$ is primitive if and only if 
there is $x\in\Gamma$ such that $[a,x]=1$. 
\begin{lemma} \label{lemmaadjsymp}
 Let $(\Al,[.,.])$ be a scalar space. Let $\Gamma$ and $\Gamma'$ be $n$-adjacent lattices in $\Al$. Assume $\Gamma$ and $\Gamma'$ 
are self-dual. Then there is a symplectic basis $(b_i)_{1\leq i\leq 2g}$ of $\Gamma$ such that $(nb_1,\frac{1}{n}b_2,b_3,\dots,b_{2g})$ 
is a symplectic basis of $\Gamma'$.
\end{lemma}
\begin{proof}
 Let $g$ be an element of $\Gamma$ which generates $\displaystyle\frac{\Gamma}{\Gamma\cap\Gamma'}\cong\frac{\Z}{n\Z}$. 
Let $b_1$ be a generator of $(\Q g)\cap\Gamma\cong\Z$. Note that $b_1$ also generates $\displaystyle\frac{\Gamma}{\Gamma\cap\Gamma'}$. 
Let us prove that $nb_1$ is a primitive element of $\Gamma'$. Assume $nb_1=k\gamma$ for some $k\in\Z$ and $\gamma\in\Gamma'$. 
For any proper divisor $n'$ of $n$, $n'b_1$ is not in $\Gamma'$. Hence $n$ and $k$ are coprime. Since $n\gamma\in\Gamma$ and $k\gamma\in\Gamma$, 
it implies that $\gamma\in\Gamma$. Thus $\gamma\in\Z(nb_1)$, and $k=\pm1$. Hence $nb_1$ is primitive in $\Gamma'$, and there is $b_2'\in\Gamma'$ 
such that $[nb_1,b_2']=1$. Set $b_2=nb_2'$. 
Let $\mathcal{B}$ be the orthogonal of $\Q b_1\oplus\Q b_2$ in $\Al$ with respect to the scalar form. 
Check that $\Gamma=(\Z b_1\oplus\Z b_2)\oplus^\perp(\mathcal{B}\cap\Gamma)$ and 
$\Gamma'=(\Z nb_1\oplus\Z \frac{1}{n}b_2)\oplus^\perp(\mathcal{B}\cap\Gamma')$. Thus $\mathcal{B}\cap\Gamma=\mathcal{B}\cap\Gamma'$ 
is a self-dual lattice in $\mathcal{B}$. Let $(b_3,\dots,b_{2g})$ be a symplectic basis of this lattice. Then the basis 
$(b_i)_{1\leq i\leq 2g}$ of $\Gamma$ satisfies the required conditions.
\end{proof}

\begin{lemma} \label{lemmastep}
 Let $(\Al,[.,.])$ be a scalar space. Let $\Gamma$ and $\Gamma'$ be $n$-adjacent admissible lattices in $\Al$. Let $\bs=(b_i)_{1\leq i\leq 2g}$ 
be a symplectic basis of $\Gamma$ such that $\bs'=(\frac{1}{n}b_1,nb_2,b_3,\dots,b_{2g})$ is a symplectic basis of $\Gamma'$. 
Let $V$ and $V'$ be the Seifert matrices associated with $\bs$ and $\bs'$ respectively. Then $V'=\Delta_n V\Delta_n$. 
If $z\Gamma\subset\Gamma'$, then $V'$ can be obtained from $V$ by an integral S-equivalence which induces the same $\tau$-class 
of isomorphisms as the congruence $V'=\Delta_n V\Delta_n$.
\end{lemma}
\begin{proof}
Let $Z$ (resp. $Z'$) be the matrix of the action of $z$ in the basis $\bs$ (resp. $\bs'$). Note that $\Delta_nZ=Z'\Delta_n$. 
By Corollary \ref{corV}, $V=ZJ$ and $V'=Z'J$. It follows that $V'=\Delta_n V\Delta_n$. 

 To prove the second statement, it suffices to prove that we can proceed as in Lemma \ref{lemmadelta}, {\em i.e.} that the coefficient 
$V_{21}$ of $V$ is divisible by $n$. Since $z\Gamma\subset(\Gamma\cap\Gamma')$, each $zb_i$ is a linear combination of 
$b_1,nb_2,b_3,\dots,b_{2g}$. It follows that $Z$ has its second row divisible by $n$. 
Hence $V=ZJ$ also has its second row divisible by $n$.
\end{proof}

\begin{proposition} \label{propSeqZ}
 Let $(M,K,\Sigma,\fs,V)$ and $(M',K',\Sigma',\fs',V')$ be two $\Z$SK-systems. Let $\xi : \Al_\Z(K)\to\Al_\Z(K')$ 
be an isomorphism which preserves the Blanchfield form. Assume $V$ and $V'$ are invertible over $\Q$. 
Then $V$ and $V'$ are related by an integral S-equivalence which canonically induces the $\tau$-class of $\xi$.
\end{proposition}
\begin{proof}
 Let $\bs=(b_i)_{1\leq i\leq 2g}$ and $\bs'=(b'_i)_{1\leq i\leq 2g}$ be bases of $\Al(K)$ and $\Al(K')$ respectively, associated with 
$V$ and $V'$. Let $P$ be the matrix of $\xi\otimes_\Z Id_\Q$ with respect to the bases $\bs$ and $\bs'$. By Proposition \ref{propcasinv}, 
$V'=PVP^t$. Let $\Gamma$ (resp. $\Gamma'$) be the lattice in $\Al(K')$ generated by the $\xi(b_i)$ (resp. by the $b_i'$). 
By Lemma \ref{lemmalattice}, $\Gamma$ and $\Gamma'$ are admissible. It is clear that $\Lambda\Gamma=\Lambda\Gamma'$. Hence, by Proposition 
\ref{propTrotter}, there is a sequence of admissible lattices $\Gamma_0=\Gamma, \Gamma_1, \dots, \Gamma_k=\Gamma'$ in $\Al(K')$, 
such that, for $1\leq j\leq k$, $\Gamma_{j-1}$ and $\Gamma_j$ are adjacent, and $z\Gamma_j\subset\Gamma_{j-1}$ 
or $z\Gamma_{j-1}\subset\Gamma_j$. 

By Lemma \ref{lemmaadjsymp}, for $1\leq j\leq k$, there are Seifert matrices $V_j$ and $\hat{V}_{j-1}$ associated with $\Gamma_j$ 
and $\Gamma_{j-1}$ respectively, and with symplectic bases of these lattices, such that $V_j=\Delta_{n_j}\hat{V}_{j-1}\Delta_{n_j}$, 
where $n_j$ is an integer or the inverse of an integer. Set $V_0=V$ and $\hat{V}_k=V'$. For $0\leq j\leq k$, the matrices $V_j$ and 
$\hat{V}_j$ are Seifert matrices associated with symplectic bases of the same lattice, and the change of basis provides an integral symplectic 
matrix $P_j$ such that $\hat{V_j}=P_jV_jP_j^t$. We have $P=P_k\Delta_{n_k}P_{k-1}\dots\Delta_{n_1}P_0$, and the $\tau$-class of the composition 
of the successive isomorphisms induced by the successive congruences is the $\tau$-class of the isomorphism induced by the congruence 
$V'=PVP^t$. Conclude with Lemma \ref{lemmastep}.
\end{proof}

\proofof{Theorem \ref{thSeqZ}} Proceed as in the proof of Proposition \ref{propSeq}. \fin

    \section{Topological realization of matrix relations} \label{sectop}

In this section, we prove Theorem \ref{thLP} and Theorem \ref{thLPZ}. 
\begin{lemma} \label{lemmaHH}
Let $M$ be a $\Q$HS (resp. $\Z$HS). Let $\Sigma$ be a genus $g$ closed connected surface embedded in $M$. 
Then $M\setminus \Sigma$ has exactly two connected components, whose closures are $\Q$HH's (resp. $\Z$HH's) 
of genus $g$.  
\end{lemma}
\begin{proof}
Any point of $M\setminus\Sigma$ can be connected to a point of $\Sigma\times[-1,1]$ in $M\setminus\Sigma$. 
Since $(\Sigma\times[-1,1])\setminus\Sigma$ 
has two connected components, $M\setminus \Sigma$ has at most two connected components. 
Let $x_1$ and $x_2$ be points of $(\Sigma\times[-1,1])\setminus\Sigma$, one in each connected component. 
If there were a path from $x_1$ to $x_2$ in $M\setminus\Sigma$, we could construct 
a closed curve in $M$ which would meet $\Sigma$ exactly once. Since $M$ is a $\Q$HS, this is not possible. 
Hence $M\setminus\Sigma$ has exactly two connected components. Let $A_1$ and $A_2$ be their closures. Note that 
$\partial A_1=\partial A_2=\Sigma$ (up to orientation). 

For $i=1,2$, we have $H_3(A_i;\Z)=0$ and $H_0(A_i;\Z)=\Z$. The Mayer-Vietoris sequence associated with 
$M=A_1\cup A_2$ yields the exact sequence:
$$H_3(M;\Z) \fl{\partial} H_2(\Sigma;\Z) \longrightarrow H_2(A_1;\Z)\oplus H_2(A_2;\Z) \longrightarrow 0.$$
The map $\partial$ is an isomorphism that identifies the fundamental classes. Thus $H_2(A_1;\Z)=H_2(A_2;\Z)=0$.
It follows that $A_1$ and $A_2$ are $\Q$HH's. Their genus is given by their boundary.

Assume $M$ is a $\Z$HS. The Mayer-Vietoris sequence associated with $M=A_1\cup A_2$ yields an isomorphism 
$H_1(\Sigma;\Z) \cong H_1(A_1;\Z)\oplus H_1(A_2;\Z)$. Hence, for $i=1,2$, $H_1(A_i;\Z)$ is torsion-free, 
thus $A_i$ is a $\Z$HH.
\end{proof}

\begin{lemma} \label{lemmacongZ}
Let $(M,K,\Sigma,\fs,V)$ and $(M',K',\Sigma',\fs',V')$ be $\Q$SK-systems. Let $P$ be an integral symplectic matrix. 
Assume $V'=PVP^t$. Then there is a null LP-surgery $(\frac{B}{A})$ in $M\setminus K$ such that $(M',K')\cong(M,K)(\frac{B}{A})$ and 
the $\tau$-class of the isomorphism $\xi:\Al(K)\to\Al(K')$ induced by this surgery is the $\tau$-class induced by $P$. 
\end{lemma}
\begin{proof}
There is a homeomorphism $h:\Sigma\to\Sigma'$, such that the matrix of the induced isomorphism $h_*:H_1(\Sigma;\Z)\to H_1(\Sigma';\Z)$ 
with respect to the bases $(f_i)_{1\leq i\leq 2g}$ and $(f'_i)_{1\leq i\leq 2g}$ is $(P^t)^{-1}$ (see \cite[theorem 6.4]{FM}). 
Let $\hat{\Sigma}$ be obtained from $\Sigma$ by adding a band glued along $\partial \Sigma$, 
so that $\hat{\Sigma}$ is homeomorphic to $\Sigma$, and contains $\Sigma$ and $K$ in its interior. Let $H=\hat{\Sigma}\times[-1,1]$ be 
a regular neighborhood of $\Sigma$. Similarly, define $\hat{\Sigma}'$ and $H'$. 
Extend $h$ to a homeomorphism $h:\hat{\Sigma}\to\hat{\Sigma}'$ , and then extend it by product with the identity to a homeomorphism $h:H\to H'$. 

Let $A=M\setminus Int(H)$ and let $B=M'\setminus Int(H')$. By Lemma \ref{lemmaHH}, $A$ (resp. $B$) is a $\Q$HH, and it is clearly null 
in $M\setminus K$ (resp. $M'\setminus K'$). The $\Q$SK-pair $(M',K')$ is obtained from $(M,K)$ 
by the surgery $(\frac{B}{A})$. Let us prove that the homeomorphism $h_{|\partial A}:\partial A \to \partial B$ preserves the Lagrangian. 

For $1\leq i\leq 2g$, let $e_i\subset\partial H$ be a meridian of $f_i$. 
The Lagrangian $\Lag_{A}$ is generated by the $\alpha_i=f_i^+ -\sum_{1\leq j\leq 2g} V_{ji}e_j$, where $f_i^+$ is the copy of $f_i$ in 
$(\Sigma\times\{1\})\subset H$. Similarly, define the $e_i'$, $(f'_i)^+$ and $\alpha_i'$. Since $h:H\to H'$ is a homeomorphism, 
$(h_{|\partial H})_*(\Lag_H)=\Lag_{H'}$. Hence $(h_{|\partial H})_*(e_i)$ is a linear combination of the $e_j'$. 
Since $\langle h(e_i),h(f_j^+)\rangle_{\partial H'}=\langle e'_i,(f'_j)^+\rangle_{\partial H'}=\delta_{ij}$, 
an easy computation gives $(h_{|\partial H})_*(e_i)=\lbp\begin{pmatrix} e'_1 & \dots & e'_{2g} \end{pmatrix}P\rbp_i$. 
It follows that $(h_{|\partial H})_*(\alpha_i)=\lbp\begin{pmatrix} \alpha'_1 & \dots & \alpha'_{2g} \end{pmatrix}(P^t)^{-1}\rbp_i$. 
Hence $(h_{|\partial A})_*(\Lag_A)=\Lag_B$.

The relation between the $h(e_i)$ and the $e_i'$ shows that the isomorphism $\xi$ induced by the surgery is in the $\tau$-class 
of isomorphisms induced by the congruence matrix $P$.
\end{proof}

The previous proof still works when $\Q$ is replaced by $\Z$. Therefore:
\begin{lemma} \label{lemmaZcongZ}
Let $(M,K,\Sigma,\fs,V)$ and $(M',K',\Sigma',\fs',V')$ be $\Z$SK-systems. Let $P$ be an integral symplectic matrix. 
Assume $V'=PVP^t$. Then there is an integral null LP-surgery $(\frac{B}{A})$ in $M\setminus K$ such that $(M',K')\cong(M,K)(\frac{B}{A})$ and 
the $\tau$-class of the isomorphism $\xi:\Al_\Z(K)\to\Al_\Z(K')$ induced by this surgery is the $\tau$-class induced by $P$. 
\end{lemma}

\begin{lemma} \label{lemmaenl}
 Let $(M,K,\Sigma,\fs,V)$ be a $\Q$SK-system. Let $W$ be an enlargement of $V$. 
Then there is a $\Q$SK-system $(M',K',\Sigma',\fs',W)$ such that $(M',K')$ can be obtained from $(M,K)$ by a single null LP-surgery, 
and the surgery and the enlargement induce the same $\tau$-class of isomorphisms from $\Al(K)$ to $\Al(K')$.
\end{lemma}
\begin{proof}
We have $$W=\begin{pmatrix} 0 & 0 & 0 \\ 1 & x & \rho^t \\ 0 & \rho & V \end{pmatrix} \textrm{ or }
\begin{pmatrix} 0 & -1 & 0 \\ 0 & x & \rho^t \\ 0 & \rho & V \end{pmatrix}.$$
We want to add a tube to $\Sigma$, whose linking numbers with the $f_i$ are given by $\rho$.
This may not be possible in $M$, so we first modify $M$ by null LP-surgeries.

Set $\rho=\begin{pmatrix} \frac{c_1}{d_1} \\ \vdots \\ \frac{c_{2g}}{d_{2g}} \end{pmatrix}$, where the $c_i$ and $d_i$ are integers,
and $d_i>0$. Consider trivial knots $J_i$, disjoint from $\Sigma$, such that $lk(J_i,f_j)=\delta_{ij} c_i$.
For each $i$, consider a tubular neighborhood $T(J_i)$ of $J_i$. By \cite[Lemma 2.5]{M2}, there are rational homology tori $A_i$ 
that satisfy:
\begin{itemize}
 \item $H_1(\partial A_i;\Z)=\Z \alpha_i \oplus \Z\beta_i$, with $<\alpha_i,\beta_i>=1$,
 \item $\beta_i=d_i\gamma_i$ in $H_1(A_i;\Z)$, where $\gamma_i$ is a curve in $A_i$,
 \item $H_1(A_i;\Z)=\Z\gamma_i\oplus\frac{\Z}{d_i\Z}\alpha_i$.
\end{itemize}
Let $N$ be the manifold obtained from $M$ by the null LP-surgeries $(\frac{A_i}{T(J_i)})$, 
where the identifications $\partial T(J_i)=\partial A_i$ identify $\alpha_i$ with a meridian of $J_i$, and $\beta_i$ with 
a parallel of $J_i$ that does not link $J_i$. We get $lk(\gamma_i,f_j)=\delta_{ij}\frac{c_i}{d_i}$.

In $N$, consider a ball $B$ disjoint from $\Sigma$ and all the $A_i$. Consider a rational homology ball $B'$ that contains
a curve $\gamma_0$ with self-linking $\lbp x-\sum_{1\leq i,j\leq 2g} lk(\gamma_i,\gamma_j)\rbp\ mod\ \Z$. Set $M'=N(\frac{B'}{B})$, and $K'=K$.
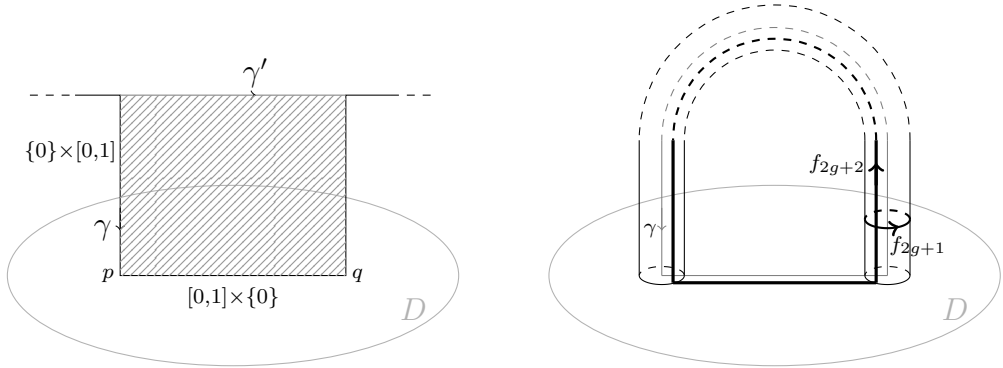
\begin{figure}[htb]
\begin{center}
\begin{tikzpicture}[scale=0.3]
\begin{scope}
\draw[color=gray!60] (5,0) ellipse (10 and 4);
\draw (-2,8) -- (0,8) -- (0,0) -- (10,0) -- (10,8) -- (12,8);
\draw[->] (0,3) -- (0,2);
\draw[dashed] (-4,8)-- (-2,8) (12,8) -- (14,8);
\draw[gray] (0,8) -- (10,8);
\draw[->] (5.9,8) -- (6,8);
\draw (0.2,0) node[left]{$\scriptstyle{p}$};
\draw (9.8,0) node[right]{$\scriptstyle{q}$};
\draw[color=gray!60] (13,-1.5) node{$D$};
\draw (0.1,2) node[left] {$\gamma$};
\draw (6,9) node{$\gamma'$};
\draw (5,-1) node{$\scriptstyle{[0,1]\times\{0\}}$};
\draw (-2.2,5.5) node{$\scriptstyle{\{0\}\times [0,1]}$};
\fill[pattern=north east lines,pattern color=gray!80] (0,8) -- (0,0) -- (10,0) -- (10,8) -- (0,8);
\end{scope}
\begin{scope} [xshift=24cm]
\draw[gray] (0,0) -- (10,0);
\draw[color=gray!60] (5,0) ellipse (10 and 4);
\foreach \a in {0,10}
{\draw[domain=0:180,dashed] plot ({\a+cos(\x)},{0.4*sin(\x)});
\draw[domain=180:360] plot ({\a+cos(\x)},{0.4*sin(\x)});}
\foreach \x in {4,6} 
  {\draw (5-\x,0) -- (5-\x,6);
   \draw[dashed] (5+\x,6) arc (0:180:\x);
   \draw (5+\x,6) -- (5+\x,0);}
\draw[gray] (0,0) -- (0,6);
\draw[gray,dashed] (10,6) arc (0:180:5);
\draw[gray] (10,6) -- (10,0);
\draw[very thick] (0.5,-0.32) -- (9.5,-0.32);
\draw[very thick] (0.5,-0.32) -- (0.5,6);
\draw[thick,dashed] (9.5,6) arc (0:180:4.5);
\draw[very thick] (9.5,6) -- (9.5,-0.32);
\draw[domain=0:180,dashed,thick] plot ({10+cos(\x)},{2.5+0.4*sin(\x)});
\draw[domain=180:360,thick] plot ({10+cos(\x)},{2.5+0.4*sin(\x)});
\draw[very thick,->] (9.5,4) -- (9.5,5);
\draw[very thick,->] (10.4,2.15) -- (10.5,2.18);
\draw[->,gray] (0,3) -- (0,2);
\draw[color=gray!60] (13,-1.5) node{$D$};
\draw (0.25,2) node[left] {$\scriptstyle{\gamma}$};
\draw (9.5,4.9) node[left] {$\scriptstyle{f_{2g+2}}$};
\draw (11.35,1.3) node {$\scriptstyle{f_{2g+1}}$};
\end{scope}
\end{tikzpicture}
\end{center} \caption{Adding a tube to $\Sigma$} \label{figtube}
\end{figure}

Define a curve $\gamma'$ in $M'$ as a band sum of the $\gamma_i$ for $0\leq i\leq 2g$, with bands outside $\Sigma$. 
Consider a disk $D$ in $\Sigma$, and two distinct points 
$p$ and $q$ in $D$. Consider an embedded band $[0,1]\times[0,1]$ in $M'$ such that $[0,1]\times\{0\} =([0,1]\times[0,1])\cap\Sigma$
is a curve from $p$ to $q$ in $D$, $[0,1]\times\{1\}=([0,1]\times[0,1])\cap\gamma'$, and the tangent vector 
to $\{0\}\times[0,1]$ at $\{0\}\times\{0\}$ is the positive normal vector of $\Sigma$ if $W$ is a row enlargement of $V$, 
and the negative one if it is a column enlargement. Figure \ref{figtube} represents the first case. 
Now set $\gamma=(\gamma'\cup\partial ([0,1]\times [0,1]))\setminus (]0,1[\times \{1\})$, and construct a surface $\Sigma'$
by adding a tube around $\gamma\setminus([0,1]\times\{0\})$ to $\Sigma$. The surface $\Sigma'$ is a Seifert surface for $K'$.  
On $\Sigma'$, consider a meridian $f_{2g+1}$ of the tube and a parallel $f_{2g+2}$ of $\gamma$ such that $<f_{2g+1},f_{2g+2}>_{\Sigma'}=1$ 
and $lk(f_{2g+2},\gamma)=x$. Note that the orientation of the meridian depends on the type of enlargement. 
The Seifert matrix associated with $\Sigma'$ with respect to the basis $(f_{2g+1},f_{2g+2},f_1,\dots,f_{2g})$ is $W$.

Since the different $\Q$HH's replaced by surgery are disjoint, they can be connected by tubes. Thus $(M',K')$ can be
obtained from $(M,K)$ by one surgery on a genus $2g$ $\Q$HH. Let $(b_i)_{1\leq i\leq 2g}$ (resp. $(b'_i)_{1\leq i\leq 2g+2}$) 
be a family of generators of $\Al(K)$ (resp. $\Al(K')$) associated with $V$ (resp. $W$). The $b_i'$ can be chosen so that the isomorphism 
$\xi:\Al(K)\to\Al(K')$ induced by the surgery satisfies $\xi(b_i)=b'_{i+2}$.
\end{proof}

\begin{lemma} \label{lemmaenlZ}
 Let $(M,K,\Sigma,\fs,V)$ be a $\Z$SK-system. Let $W$ be an integral enlargement of $V$. 
Then there is a $\Z$SK-system $(M,K,\Sigma',\fs',W)$ and the enlargement induces the $\tau$-class of the identity of $\Al_\Z(K)$.
\end{lemma}
\begin{proof}
 In the previous proof, replace $\Q$ by $\Z$, and remove the definition of the surgery, since any integral linking can be realised 
in any $\Z$HS.
\end{proof}

\proofof{Theorem \ref{thLP}}
Let $V$ and $V'$ be Seifert matrices associated with $(M,K)$ and $(M',K')$ respectively. 
By Theorem \ref{thSeq}, $V'$ can be obtained from $V$ by a sequence of enlargements, reductions, and integral symplectic 
congruences which induces the $\tau$-class of the isomorphism $\xi$. This provides a finite sequence $V_1,V_2,..,V_n$ of Seifert matrices 
such that $V_1=V$, $V_n=V'$, and $V_{i+1}$ is obtained from $V_i$ by one of these equivalences. By Lemma \ref{lemmarealmat}, 
for each $i$, we can fix a $\Q$SK-system $\CS_i$ with Seifert matrix $V_i$. To see that $\CS_{i+1}$ can be obtained from $\CS_i$ 
by one, or two successive, null LP-surgeries, which induce the required $\tau$-class of isomorphisms, apply Lemma \ref{lemmacongZ} 
in the case of a (possibly trivial) congruence, and apply Lemma \ref{lemmaenl} in the case of an enlargement or a reduction.
\fin 

Similarly, Theorem \ref{thLPZ} can be deduced from Theorem \ref{thSeqZ} and Lemmas \ref{lemmarealmat}, \ref{lemmaZcongZ}, and~\ref{lemmaenlZ}.

    \section{Sequences of LP-surgeries} \label{secfin}

In this section, we prove Proposition \ref{propfin}. 

\begin{lemma} \label{lemmaex}
 There exist two knots in $S^3$ which have isomorphic rational Blanchfield forms, and different integral 
Alexander modules.
\end{lemma}
\begin{proof}
 In $S^3$, consider a knot $K$ with Seifert matrix $\begin{pmatrix} -1 & 0 \\ 1 & 2 \end{pmatrix}$, and a knot $K'$
with Seifert matrix $\begin{pmatrix} 3 & 1 \\ 2 & 0 \end{pmatrix}$. Their Alexander modules have presentation matrices 
$\begin{pmatrix} 1-t & -1 \\ t & 2t-2 \end{pmatrix}$ and $\begin{pmatrix} 3t-3 & t-2 \\ 2t-1 & 0 \end{pmatrix}$.
Both have Alexander polynomial $\Delta(t)=(2t-1)(2-t)$. Since it is the product of two dual non symmetric prime polynomials,
their rational Blanchfield forms are isomorphic (see \cite[Lemma 3.6]{M1}). 
But $K$ has integral Alexander module $\frac{\Z[t^{\pm1}]}{(\Delta(t))}$, whereas the integral Alexander module of $K'$ has a non 
trivial second elementary ideal (the $k$-th elementary ideal associated with a $\Z[t^{\pm1}]$-module 
is the ideal of $\Z[t^{\pm1}]$ generated by the minors of size $n-k+1$ of a presentation matrix with $n$ generators of the module, 
see \cite[Chapter 6]{Lick}). Indeed, this ideal is generated by $(t-2)$ and $(2t-1)$ in $\Z[t^{\pm1}]$, so the evaluation at $t=-1$ 
maps it onto $3\Z$.
\end{proof}

\proofof{Proposition \ref{propfin}} 
Consider the $\Q$SK-pairs $(S^3,K)$ and $(S^3,K')$ of Lemma \ref{lemmaex}. 
By Theorem \ref{thLP}, $(S^3,K')$ can be obtained from $(S^3,K)$ by a finite sequence of null LP-surgeries. 
Suppose we can restrict to a single surgery $(\frac{B}{A})$. Then $A$ and $B$ would be $\Q$HH's embedded in a $\Z$HS. 
It follows from Lemma \ref{lemmaHH} that $A$ and $B$ would be $\Z$HH's. 
Thus, by Lemma \ref{lemmaintA}, the surgery would preserve the integral Alexander module.\fin

\def\cprime{$'$}
\providecommand{\bysame}{\leavevmode ---\ }
\providecommand{\og}{``}
\providecommand{\fg}{''}
\providecommand{\smfandname}{\&}
\providecommand{\smfedsname}{\'eds.}
\providecommand{\smfedname}{\'ed.}
\providecommand{\smfmastersthesisname}{M\'emoire}
\providecommand{\smfphdthesisname}{Th\`ese}

\end{document}